\documentclass[pdflatex,sn-mathphys-num]{sn-jnl}

\usepackage{graphicx}%
\usepackage{multirow}%
\usepackage{amsmath,amssymb,amsfonts}%
\usepackage{amsthm}%
\usepackage{mathrsfs}%
\usepackage[title]{appendix}%
\usepackage{xcolor}%
\usepackage{textcomp}%
\usepackage{manyfoot}%
\usepackage{booktabs}%
\usepackage{algorithm}%
\usepackage{algorithmicx}%
\usepackage{algpseudocode}%
\usepackage{listings}%

\theoremstyle{thmstyleone}%
\newtheorem{theorem}{Theorem}
\newtheorem{proposition}[theorem]{Proposition}%

\theoremstyle{thmstyletwo}%
\newtheorem{example}{Example}%
\theoremstyle{thmstylethree}%
\newtheorem{remark}{Remark}%
\theoremstyle{thmstylethree}%
\newtheorem{lemma}{Lemma}%
\theoremstyle{thmstylethree}%
\newtheorem{corollary}{Corollary}%
\theoremstyle{thmstylethree}%
\newtheorem{definition}{Definition}%

\usepackage{tikz}
\usetikzlibrary{arrows.meta} 
\usepackage{caption} 

\begin{document}

\title[Article Title]{Generalized geometric constants related to Birkhoff orthogonality in Banach spaces}


\author[1]{\fnm{Junxiang} \sur{Qi}}\email{Y25060009@stu.aqnu.edu.cn}

\author[1]{\fnm{Qian} \sur{Li}}\email{Y23060030@stu.aqnu.edu.cn}
\equalcont{These authors contributed equally to this work.}

\author[1]{\fnm{Zhouping} \sur{Yin}}\email{yzp@aqnu.edu.cn}
\equalcont{These authors contributed equally to this work.}

\author*[1]{\fnm{Qi} \sur{Liu}}\email{liuq67@aqnu.edu.cn}
\equalcont{These authors contributed equally to this work.}

\author[2]{\fnm{Jiaye} \sur{Bi}}\email{bijiaye@mail2.sysu.edu.cn}
\equalcont{These authors contributed equally to this work.}

\author[3]{\fnm{Yuankang} \sur{Fu}}\email{ykfu@fafu.edu.cn}
\equalcont{These authors contributed equally to this work.}

\author[2]{\fnm{Yongjin} \sur{Li}}\email{stslyj@mail.sysu.edu.cn}
\equalcont{These authors contributed equally to this work.}

\affil*[1]{\orgdiv{School of Mathematics and Statistics}, \orgname{Anqing Normal University}, \orgaddress{\city{Anqing}, \postcode{246133}, \country{China}}}

\affil[2]{\orgdiv{Department of Mathematics}, \orgname{Sun Yat-sen University}, \orgaddress{\city{Guangzhou}, \postcode{510275}, \country{P. R. China}}}

\affil[3]{\orgdiv{College of Computer and Information Sciences}, \orgname{Fujian Agriculture and Forestry University}, \orgaddress{\city{Fuzhou}, \postcode{350002}, \country{P. R. China}}}


\abstract{	In this paper, based on Birkhoff orthogonality, we introduce two geometric constants \(A_{t}^{B}(X)\) and \(D_{t}^{B}(X)\) in Banach spaces, which generalize the existing geometric constants.
	Moreover, one of these constants is investigated in connection with the newly introduced orthogonality relation.
	We systematically study the basic properties of the two constants, including their upper and lower bounds, and establish equivalent characterizations for a Banach space to be uniformly non-square.
	Additionally, we explore the relationship between \(D_{t}^{B}(X)\) and the modulus of convexity.
	Finally, we present several applications of the two newly proposed geometric constants.
}




\maketitle
\section{Introduction}
In contemporary research, specific geometric constants have been introduced and explored in academic research. These constants can quantitatively describe the geometric properties of spaces. Among them, orthogonality, as a core concept in Euclidean geometry, is not only an important part of the axiom system but also provides key support for many basic theorems. For more papers on geometric constants, refer to \cite{9,13,14,15,16,17,18,19,20,22}. 

Geometric constants endowed with orthogonality present an interesting research topic. Studying geometric constants with orthogonality forms requires considering orthogonal relations, which in turn helps explore whether local variables satisfying specific orthogonality conditions can replace global variables when analyzing the geometric properties of Banach spaces. For more papers on geometric constants related to orthogonality, refer to \cite{4,27,30,31,32}.

In the theory of Banach spaces, finding suitable generalized forms for this core concept has always been a research focus in the field of geometric analysis. Scholars have successively proposed various generalized definitions such as Roberts orthogonality, Birkhoff orthogonality, and isosceles orthogonality. 

In 1934, Roberts \cite{1} proposed Roberts orthogonality: Let \(X\) be a real Banach space, \(x\) is said to be Roberts orthogonal to \(y\) for any \(x, y \in X\), (\(x \perp_R y\)) if and only if
$$
 \forall t \in \mathbb{R},\quad \|x+t y\|=\|x-t y\|
$$

Birkhoff first introduced the concept of Birkhoff orthogonality in \cite{2}: Let \(X\) be a real Banach space, we define \(x\)
to be Birkhoff orthogonal to \(y\)
(written as \(x \perp_B y\)
) if and only if
$$
\|x+t y\| \geqslant\|x\|, \quad \forall t \in \mathbb{R}
$$

James \cite{3} show that isosceles orthogonality in 1947: \(x\) is said to be isosceles orthogonal to \(y\) (\(x \perp_I y\)) if and only if

$$
\|x+y\|=\|x-y\| .
$$

The constant $A_2(X)$, has been extensively investigated by Baronti et al. \cite{11}. which is defined as follows:
$$
A_2(X)=\sup \left\{\frac{\|x+y\|+\|x-y\|}{2}: x, y \in S_X\right\}.
$$

On the unit sphere $S_X$, suppose that the points $y$ and $-y$ are antipodal, then we can obtain 2, $\|x+y\|$ and $\|x-y\|$ can be interpreted as the side lengths of the triangle $T_{x y}$, the parameter $A_2(X)$ can be characterized as the supremum of the arithmetic mean of the side lengths of $T_{x y}$.

Birkhoff orthogonality is intimately associated with the geometric structure and constitutes the primary motivation for the introduction of the parameter associated with Birkhoff orthogonality, building upon the aforementioned coefficient \cite{5}:
$$
A_{2}(X,B)=\sup \left\{\frac{\|x+y\|+\|x-y\|}{2}: x, y \in S_X,x \perp_B y\right\}.
$$

In addition, some properties of constant are as follows:\\

\rm{(1)} For any Banach space $X$, the inequality $\sqrt{2} \leq A_{2}(X,B) \leq 2$ holds;

\rm{(2)} Suppose that $X$ is a Radon plane, then we have $A_{2}(X,B) \leq \frac{3}{2}$;

\rm{(3)} For a given Banach space $X$, the equality $A_{2}(X,B)=2$
holds if and only if $X$
does not satisfy the uniformly non-square property.

\rm{(4)} Assume that $X$ is a Radon plane, $ A_2(X,B)=\frac{3}{2}$ if and only if the unit sphere $S_{X}$ is an affine-regular hexagon.

A Banach space $X$ is said to be uniformly non-square \cite{10} if there exists a constant $\delta \in(0,1)$
such that for all $x, y \in S_X$, at least one of the inequalities 
$\frac{\|x+y\|}{2} \leq 1-\delta$
or $\frac{\|x-y\|}{2} \leq 1-\delta$
hold. And defined the James constant \cite{10}:

$$
J(X)=\sup \left\{\min \{\|x+y\|,\|x-y\|\}: x, y \in S_X\right\} .
$$

Papini and Baront\i \cite{6} introduces the constant in any space $X$,
$$
J^B(X)=\sup \left\{\min \{\|x+y\|,\|x-y\|\}: x, y \in S_X, x \perp_B y\right\}.
$$
	
	Ji  and Wu  introduced the constants $D(X)$ and $D^{\prime}(X)$ in \cite{28} to measure the distinction between Birkhoff orthogonality and isosceles orthogonality.

	$$
	\begin{aligned}
	D(X) & =\inf \left\{\inf _{\lambda \in \mathbb{R}}\|x+\lambda y\|: x, y \in S_X, x \perp_I y\right\} \\
	D^{\prime}(X) & =\sup \left\{\|x+y\|-\|x-y\|: x, y \in S_X, x \perp_B y\right\} .
\end{aligned}
$$

Papini and Wu \cite{7} introduced a new constant in Banach space $X$,
$$\begin{aligned}
	BR(X)&=\sup_{t>0}\bigg\{\frac{\|x+t y\|-\|x-t y\|}{t}:x, y \in S_X, x \perp_B y\bigg\}\\
	&=\sup_{t>0}\bigg\{\|t x+ y\|-\|t x- y\|:x, y \in S_X, x \perp_B y\bigg\}.
\end{aligned}$$

Moreover, the various properties of these constants are given in:

\rm{(1)} $0 \leq BR(X) \leq 1$; 

\rm{(2)} $1 \leq J^B(X) \leq 2$;

\rm{(3)} $BR(X)=0$ if and only if $X$ is an i.p.s.;

In \cite{8}, Fitzpatrick and Reznick proposed the skewness $s(X)$ of the space $X$, a parameter that characterizes the asymmetry of the norm:
$$
s(X)=\sup\left\{\lim_{\lambda\to 0^+}\frac{\|x+\lambda y\|-\|y+\lambda x\|}{\lambda}:x,y\in S_X\right\}.
$$

The constant $s(X)$ is bounded between 0 and 2 for all Banach spaces $X$. Additionally, prior studies have demonstrated the equivalence relations: $s(X)=0$
corresponds exactly to $X$ being a Hilbert space, and $s(X)=2$ is equivalent to $X$ not being uniformly nonsquare. 

	Suppose $X $ is a Banach space, for every  $\varepsilon \in[0,2] $, we define the modulus of convexity of the norm $\|\cdot\|$
	to be \rm{\cite{322}}
	$$
	\delta_{X}(\varepsilon)=\inf \left\{1-\left\|\frac{x+y}{2}\right\|: x, y \in B_{X},\|x-y\| \geq \varepsilon\right\} .
	$$
	
	For every $ \varepsilon \in(0,2]$, we say that the norm $ \|\cdot\|$ is uniformly
	convex if $\delta_{X}(\varepsilon)>0 $ , and $(X,\|\cdot\|)$ is
	called a uniformly convex space.

The motivation of this paper is to generalize existing geometric constants related to Birkhoff orthogonality by introducing the parameter $t>0$, and construct two new geometric constants $A^B_{t}(X)$ and $D_{t}^{B}(X)$.It inherits the core of asymmetry characterization from skew-type constants, and unifies the generalization logic of such constants through the parameter $t$, enabling the constants to adapt to more vector combination scenarios. 

On one hand, the flexibility of the parameter $t$ is utilized to cover vector combination scenarios with different weights, and normalization ensures that the constants are intrinsic properties of the space. On the other hand, we establish their equivalent characterization of uniform nonsquareness and quantitative relationship with the modulus of convexity. This enriches the library of quantitative tools for the geometric properties of Banach spaces and improves the characterization link from orthogonality to the core structure of the space.

The organization of this article is outlined below:

In Section 2, we introduce a new constant $A^B_{t}(X)$, and in the meantime, we investigate several of its fundamental properties, such as its range and comparisons with the classic constant $J^B(X)$. In addition, we gave examples in specific spaces and reported findings concerning uniform non-squareness and its relationship with the skewness $s(X)$.

In Section 3, we consider the $D_{t}^{B}(X)$, a new geometric constant. Then we present several basic properties of this constant, including its upper and lower bounds as well as its connection to the uniform non-squareness of Banach spaces. We also establish the relation between $D_{t}^{B}(X)$ and the modulus of convexity.

In Section 4, we explore the two constants the relationships between $A_{2}(X,B)$, and characterize the inner product space. In particular, By means of the geometric parameter $A^B_{t}(X)$, we obtain a characterization of Radon planes whose unit spheres is a affine-regular hexagon.
\section{The constant $A^B_{t}(X)$}
We define the constant $A^B_{t}(X)$ by extending and generalizing the form of the existing Birkhoff orthogonality-related geometric constant $A_{2}(X,B)$ by introducing a parameter $t>0$. We suppose that $X$ is a Banach space with
dimension at least 2.
 Next we start by outlining the key definitions: for fixed $t>0$, 
$$
\begin{aligned}
	&  A^B_{t}(X)=\sup \left\{\frac{\| x+ty\|+\|tx-y\|}{2}: x, y \in S_X, x \perp_B y\right\}.
\end{aligned}
$$

Specially, let $t=1$, then we have

$$
\begin{aligned}
	A^B_{1}(X) =\sup \left\{\frac{\| x+y\|+\|x-y\|}{2}: x, y \in S_X, x \perp_B y\right\}=A_{2}(X,B).
\end{aligned}
$$

To begin with, we use a lemma to derive the bounds of $A^B_{t}(X)$.

\begin{lemma}\cite{32}\label{11}
(1) The function $f(t)=\|x+t y\|+\|t x-y\|$ is a convex function of $t$ on $\mathbb{R}$.\\
	(2) The function $g(t)=\|t x+y\|+\|x-t y\|$ is a convex function of $t$ on $\mathbb{R}$.
\end{lemma}

\begin{proposition}\label{P1}
	Suppose that $X$ is a Banach space, then $\sqrt2 \min\{1,t\}\leq A^B_{t}(X) \leq 1+t$.
\end{proposition} 
\begin{proof}  First of all, the function $f: \mathbb{R} \rightarrow \mathbb{R}$ defined as	
	$$
	f(t)=\|x+t y\|+\|t x-y\| .
	$$
	For any $x, y \in S_X$, we obtain that
	$$
	f(0)=\|x\|+\|y\|=2
	$$
	and
	$$
	f(1)=f(-1)=\|x+y\|+\|x-y\| \geqslant 2 .
	$$
	By the Lemma \ref{11}, we know that $f(t)$ is a convex function, so $f\left(t\right) \geqslant f(1)$ for $t \geqslant 1$ and using the same technique, $g\left(\frac{1}{t}\right) \geqslant g(1)$ for $\frac{1}{t} \geqslant 1$.
	
	When $t \geqslant 1$, for any $x, y \in S_X$, we have the right inequality hold,
	$$\begin{aligned}
		\frac{\| x+t y\|+\|t x- y\|}{2}&\leq
		\frac{\| x\|+\|t y\|+\|t x\|+\| y\|}{2}
		\\&=1+t.
	\end{aligned}
	$$
	
	On the other hand,  for any \( x, y \in S_X \), $x \perp_B y$ and \( t \geqslant 1 \), we can obtain 
	$$\begin{aligned}
		\frac{\left\|x+t y\right\|+\left\|t x-y\right\|}{2} &\geqslant \frac{\|x+y\|+\|x-y\|}{2}\\&\geqslant \sqrt{\|x+y\|\|x-y\|},
	\end{aligned}$$
	thus, we have
	$$
	\frac{\| x+t y\|+\|tx-y\|}{2} \geqslant  \sqrt2 .
	$$
	Similarly, for any \( x, y \in S_X \) and \( t \in (0,1] \),  yielding the following inequality,
	$$\begin{aligned}
		\frac{t\left(\|\frac{1}{t}x+ y\|+\|x-\frac{1}{t}y\|\right)}{2t} &\geqslant\frac{\|x+y\|+\|x-y\|}{2}\\&\geqslant \sqrt{\|x+y\|\|x-y\|},
	\end{aligned}$$
from which it follows that
	$$
	\frac{\| x+t y\|+\|tx-y\|}{2} \geq \sqrt{2} t,
	$$ so we complete the proof.
\end{proof} 

Then we analyse the relation between
the constant  $A^B_{t}(X)$ and the famous constant $J^B(X)$.
\begin{proposition}\label{P2} Suppose that $X$ is a Banach space. Then 
	$$
	\max \{1, t\} J^B(X)-|1-t| \leqslant A^B_{t}(X) \leqslant \frac{1}{2} J^B(X)+1+|1-t| .
	$$
\end{proposition} 

\begin{proof} For $x, y \in S_X$  such that $
	x \perp_B y
	$ and $t>0$, we have
	
	$$
	\begin{aligned}
		\min \{\|x+y\|,\|x-y\|\} 
		& \leqslant \min \{\| x+t y\|+|1-t|,\|t x- y\|+|1-t|\} \\
		& \leqslant \frac{\| x+t y\|+\|t x- y\|}{2}+|1-t|,
	\end{aligned}
	$$
	and
	$$
	\begin{aligned}
		t \min \{\|x+y\|,\|x-y\|\} &  \leqslant \min \{\| x+t y\|+|1-t|,\|t x- y\|+|1-t|\} \\
		& \leqslant \frac{\| x+t y\|+\|t x-y\|}{2}+|1-t| .
	\end{aligned}
	$$
	So we can obtain $\max \{1, t\} J^B(X)-|1-t| \leqslant A^{B}_{t}(X)$.
	
	On the other hand,  
	for the above of $x, y $, 
	we have
	$$
	\begin{aligned}
		\frac{\| x+t y\|+\|t x- y\|}{2} & \leqslant \frac{\|x+y\|+\|x-y\|}{2}+|1-1|+|1-t| \\
		& \leq\frac{J^B(X)}{2}+1+|1-t|,
	\end{aligned}
	$$
	as desired.
\end{proof} 
Focusing on the relevant properties of the constant $A_2(X,B)$, next, we will investigate the connection between $A^B_{t}(X)$  and $ A_2(X,B)$, we first introduce a Lemma.
\begin{lemma}\label{2.2}\cite{34}
For nonzero vectors $x$ and $y$ in a normed space $X$ it is true that

$$
\|x+y\| \leq\|x\|+\|y\|-\left(2-\left\|\frac{x}{\|x\|}+\frac{y}{\|y\|}\right\|\right) \min (\|x\|,\|y\|).
$$
\end{lemma}

\begin{proposition}\label{P2} Suppose that $X$ is a Banach space. Then 
	$$
	\min \{1, t\} A_2(X,B) \leqslant A^B_{t}(X) \leqslant  1+t-(2-A_2(X,B))\min\{1,t\}.
	$$
\end{proposition} 

\begin{proof} 
	
		By the method analogous to that for Proposition \ref{P1}, we have the following conclusion.
	When $t \geqslant 1$, for any $x, y \in S_X$, we have
	$$
	\frac{\left(\left\|x+t y\right\|+\left\|t x-y\right\|\right)}{2} \geqslant \frac{\|x+y\|+\|x-y\|}{2} .
	$$
	So we have 
	$$
	A^B_{t}(X)\geqslant A_2(X,B) .
	$$
	
	Similarly, $0<t \leqslant 1$, for any $x, y \in S_X$, we have
	
	$$
	\frac{t\left(\left\|\frac{1}{t} x+y\right\|+\left\|x-\frac{1}{t} y\right\|\right)}{2 t} \geqslant \frac{\|x+y\|+\|x-y\|}{2} ,
	$$
	which implies that
	$
	\frac{A^B_{t}(X)}{t} \geqslant A_2(X,B)
	$.

	On the other hand, by the Lemma \ref{2.2}, we have 
	$$\| x+t y\|\leq \|x\|+\|ty\|-\left(2-\left\|\frac{x}{\|x\|}+\frac{ty}{\|ty\|}\right\|\right) \min (\|x\|,\|ty\|)$$
	$$\|t x- y\|\leq\|tx\|+\|y\|-\left(2-\left\|\frac{tx}{\|tx\|}-\frac{y}{\|y\|}\right\|\right) \min (\|tx\|,\|y\|).
$$	
Since $x, y \in S_X$, we can obtain
	$$
	\begin{aligned}
		\frac{\| x+t y\|+\|t x- y\|}{2} & \leqslant \frac{2(1+t)-\left(2-\|x+y\|\right) \min \{1,t\}-\left(2-\|x-y\|\right) \min\{1,t\}}{2} \\
		& =1+t-\Big(2-\frac{\|x+y\|+\|x-y\|}{2}\Big) \min\{1,t\}\\&\leq 1+t-(2-A_2(X,B))\min\{1,t\},
	\end{aligned}
	$$
	as desired.
\end{proof} 
The utility of a new geometric constant is demonstrated by its ability to provide new insights
or criteria for important geometric properties. Here, in Banach spaces, we discusses $A^B_{t}(X)$ uniform non-squareness.

Note that a Banach space $X$ is not uniformly non-square if and only if
$A_2(X,B)=2$ \cite{5}, by the above Proposition \ref{P2}, we deduce the
following Proposition.

\begin{proposition}\label{2.3}
	Let $X$ be a Banach space. Then $A^B_{t}(X)=1+t$ if and only if $X$ is not uniformly non-square.
\end{proposition}
\begin{proof} 
	By  Proposition \ref{P2}, 
	when  $A^B_{t}(X)=1+t$ we deduce that
	$$
	1+t\leqslant 1+t-(2-A_2(X,B))\min\{1,t\} .
	$$
	This shows that $A_2(X,B)\geq2$. It should be noted that the upper bound of the constant is 2, and combined with the conclusion that  $A_2(X,B)=2$ if and only if the space \( X \) is not uniformly \cite{5}, as desired.
\end{proof}

In what follows, we put forward the proposition below to demonstrate that the value $1+t$ in the preceding implication cannot be substituted with a lesser number.

\begin{proposition}
	For any $\varepsilon>0$, $t>0$, there exist a uniformly convex and also uniformly smooth Banach spaces $X$ such that 
	$A^B_{t}(X)\geq1+t-\varepsilon.$
\end{proposition}

\begin{proof} Let $\varepsilon>0$. First, observe that 
	$$
	\lim_{p \to \infty} 2^{-\frac{1}{p}}(|1-t|^p+|1+t|^p)^{1/p}=1+t.
	$$
	Thus, there exists some $p>1$ 
	for which
	 $$2^{-\frac{1}{p}}(|1-t|^p+|1+t|^p)^{1/p}\geq1+t-\varepsilon.$$
	 Next, take the uniformly convex and also uniformly smooth space $X=\mathbb{R}^2$
	 equipped with the  norm
	$$
\left\|\left(x_1, x_2\right)\right\|=\left(\left|x_1\right|^p+\left|x_2\right|^p\right)^{1 / p},\quad 1<p<\infty.
	$$ 
	
	 Suppose that $x=\left(2^{-\frac{1}{p}}, 2^{-\frac{1}{p}}\right)$ and $y=\left(-2^{-\frac{1}{p}}, 2^{-\frac{1}{p}}\right)$, then we can see that $x, y\in S_X$ such that $x \perp_B y$. Thus, from the definition of $A^B_{t}(X)$ that we have
	$$
	\begin{aligned}
		A^B_{t}(X)&\geq\frac{\|x+ty\|+\|tx- y\|}{2}\\&=2^{-\frac{1}{p}}\cdot\frac{(|1-t|^p + |1+t|^p)^{1/p} + (|t+1|^p +|t-1|^p)^{1/p}}{2}\\&=2^{-\frac{1}{p}}(|1-t|^p+|1+t|^p)^{1/p}\\&\geq1+t-\varepsilon.
	\end{aligned}
	$$
	This concludes the proof.
\end{proof}

\begin{lemma}\label{1}
	\rm{(\cite{21})} Suppose that $X$ is a Banach space, $x\in X$ and $x\neq0$. Then, for every $y\in X$, we have the function
	
	$$
	g(\lambda)=\frac{\|x + \lambda y\|-\|x\|}{\lambda},
	$$
	whose domain is $\mathbb{R}\setminus\{0\} $
	and range is contained in $\mathbb R$
	is non-decreasing on its entire domain.
\end{lemma}
\begin{proposition}
	If  $X $ is a Banach space and $t>0$. Then
	
	$$ 
	s(X)\geq \frac{2(A^B_{t}(X)+\lambda-1-\lambda^2-t)}{\lambda(1 +\lambda)}
	$$
	for any $\lambda$ with $0<\lambda\leq 1$. 
\end{proposition}

\begin{proof}
	Let $0<\lambda_0\leq1$, Since $t>0$ and $s(X)\geq 0$, we can assume that
	
	$$ 
	A^B_{t}(X)+\lambda_0-1-\lambda_0^2-t>0.
	$$ 
	Suppose further that
	$\varepsilon>0$
	and
	$$ 
	\varepsilon<A^B_{t}(X)+\lambda_0-1-\lambda_0^2-t.$$
	By the definition of $A^B_{t}(X)$, we have
	$u, v\in S_X$ 
	with
	$u\perp_Bv$ 
	and 
	\begin{equation} 
		A^B_{t}(X)-\varepsilon<\frac{\|u+tv\|+\|tu-v\|}{2}.
	\end{equation}  
	
	By applying the triangle inequality, we deduce that
	\begin{equation}
		\|u+\lambda_0 v\|\leq\|u\|+\lambda_0\|v\|=1+\lambda_0,
	\end{equation} 
	\begin{equation}
		\|v-\lambda_0 u\|\leq\|v\|+\|\lambda_0u\|=1+\lambda_0,
	\end{equation}
	\begin{equation}
		\|u+tv-tv+\lambda_0 v\|\geq\|u +tv\|-t+\lambda_0,
	\end{equation} 
	\begin{equation}
		\|v-\lambda_0u+tu-tu\|\geq\|tu-v\|-t+\lambda_0.
	\end{equation} 
	
	Next, we divide the discussion into the following two cases:
	
	Case 1:
	$\|u+\lambda_0 v\|\geq \|v-\lambda_0 u\|$. 
	
	Let $w=u+\lambda_0 v,z=v-\lambda_0 u$. Then we have
	\begin{equation}
		\|w-\lambda_0z\|=\|u+\lambda_0^2u \| =1 + \lambda_0^2
	\end{equation}
	\begin{equation}
		\|z+\lambda_0 w\|=\|v+\lambda_0^2v \|=1+\lambda_0^2.
	\end{equation}
	
	By (4), (5), (6), (7) and Lemma \ref{1}, we derive that for any $\lambda$ with $0<\lambda<\lambda_0$,
	
	$$ 
	\begin{aligned}
		\frac{\|w+\lambda z\|-\|w\|}{\lambda}&\geq \frac{\|w - \lambda_0 z\| - \|w\|}{-\lambda_0} \\&= \frac{\|w\| -\|w-\lambda_0 z\|}{\lambda_0}\\&\geq \frac{\|u+ tv\|-1-t-\lambda_0^2+\lambda_0}{\lambda_0}
	\end{aligned}
	$$ 
	
	and
	
	$$
	\begin{aligned}
		\frac{\|z +\lambda w\|-\|z\|}{\lambda}&\leq \frac{\|z + \lambda_0 w\| - \|z\|}{\lambda_0}\\&\leq\frac{-\|tu-v\|+1+\lambda_0^2 +t-\lambda_0}{\lambda_0}.
	\end{aligned}
	$$ 
	Since $ \|w\|\geq\|z\|$, then from the above inequalities and (2.1) that we have
	
	$$ 
	\begin{aligned}
		&\frac{\|w +\lambda z\|-\|z +\lambda w\|}{\lambda}\\&\geq\frac{\|u+tv\|+\| tu - v\| -2(1+\lambda_0^2-\lambda_0+t)}{\lambda_0}\\&\geq\frac{2(A^B_{t}(X)+\lambda_0-1-\lambda_0^2-t-\varepsilon)}{\lambda_0}.
	\end{aligned}
	$$ 
	From (3), (4) and $u\perp_Bv$, we have

	$$
	\|w\|\geq\|u +tv\|-t+\lambda_0 \geq \|tv\|-t +\lambda_0 = \lambda_0> 0,
	$$ 
	$$ 
	\|z\|\geq\|tu-v\|-t+\lambda_0\geq \|tu\|-t +\lambda_0 = \lambda_0> 0.
	$$ 
	Take $x =\frac{w}{\|w\|}$ and $y =\frac{z}{\|z\|}$. 
	Notice that
	$$ 
	A^B_{t}(X)+\lambda_0-1-\lambda_0^2-t-\varepsilon>0.
	$$ 
	 From (1) and (2) that we derive
	$$
	\frac{\|x + \lambda y\| - \|y +\lambda x\|}{\lambda} \geq \frac{2(A^B_{t}(X)+\lambda_0-1-\lambda_0^2-t-\varepsilon)}{\lambda_0(1 +\lambda_0)}.
	$$ 
	Case 2: $\|u+\lambda_0 v\|< \|v-\lambda_0 u\|$. 
	
	Let $w=v+\lambda_0(-u),z=-u-\lambda_0v$. Then we have
	\begin{equation}
		\|w-\lambda_0z\|=\|v+\lambda_0^2v\| =1 + \lambda_0^2,
	\end{equation} 
	\begin{equation}
		\|z+\lambda_0 w\|=\|-u-\lambda_0^2u \|=1+\lambda_0^2.
	\end{equation}
	By (4), (5), (8), (9) and Lemma \ref{1}, we obtain that for any $\lambda$ with $0<\lambda<\lambda_0$,
	
	$$ 
	\begin{aligned}
		\frac{\|w+\lambda z\|-\|w\|}{\lambda}&\geq \frac{\|w - \lambda_0 z\| - \|w\|}{-\lambda_0} \\&= \frac{\|w\| -\|w-\lambda_0 z\|}{\lambda_0}\\&\geq \frac{\|u+ tv\|-1-t-\lambda_0^2+\lambda_0}{\lambda_0}
	\end{aligned}
	$$ 
	
	and
	
	$$
	\begin{aligned}
		\frac{\|z +\lambda w\|-\|z\|}{\lambda}&\leq \frac{\|z + \lambda_0 w\| - \|z\|}{\lambda_0}\\&\leq\frac{-\|tu-v\|+1+\lambda_0^2 +t-\lambda_0}{\lambda_0}.
	\end{aligned}
	$$ 
	Since $ \|w\|<\|z\|$, then, from the above inequalities and (2.1) that we have
	$$ 
	\begin{aligned}
		&\frac{\|w +\lambda z\|-\|z +\lambda w\|}{\lambda}\\&\geq\frac{\|u+tv\|+\| tu - v\| -2(1+\lambda_0^2-\lambda_0+t)}{\lambda_0}\\&\geq\frac{2(A^B_{t}(X)+\lambda_0-1-\lambda_0^2-t-\varepsilon)}{\lambda_0}.
	\end{aligned}
	$$ 
	From (3), (4) and $u\perp_Bv$, we have
	$$
	\|w\|\geq\|u +tv\|-t+\lambda_0 \geq \|tv\|-t +\lambda_0 = \lambda_0> 0,
	$$ 
	$$ 
	\|z\|\geq\|tu-v\|-t+\lambda_0\geq \|tu\|-t +\lambda_0 = \lambda_0> 0.
	$$ 
	Take $x =\frac{w}{\|w\|}$ and $y =\frac{z}{\|z\|}$. 
	Notice that
	$$ 
	A^B_{t}(X)+\lambda_0-1-\lambda_0^2-t-\varepsilon>0.
	$$ 
	Then it follows from (1), (2) that we derive
	$$
	\frac{\|x + \lambda y\| - \|y +\lambda x\|}{\lambda} \geq \frac{2(A^B_{t}(X)+\lambda_0-1-\lambda_0^2-t-\varepsilon)}{\lambda_0(1 +\lambda_0)}.
	$$ 
Based on the concept of $s(X)$, when $\varepsilon \to 0^+$, the inequality is proved .
\end{proof} 

\section{The constant $D_{t}^{B}(X)$}
Lemma \ref{3.2}, which will be used in this chapter, establishes a result concerning skew equality that is both meaningful and interesting. In particular, it can be used to characterize inner product spaces.
Next, we  introduce an orthogonality notion extending isosceles orthogonality, which is also skew.

\begin{definition}
Suppose that  $X$ is a real normed linear space of dimension at least two. For any $x, y \in X$, then $x$ is said to be skew isosceles orthogonal to $y$ 
if (see Figure 1) for the fixed parameter $t$,
$$  \| x+t y\|=\|t x- y\|. $$

\end{definition}

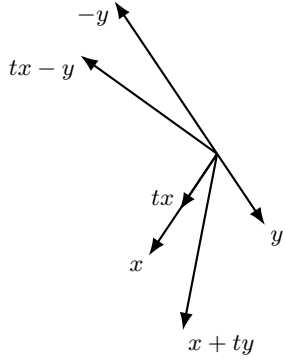
\begin{figure}[H]
	\begin{center}
		\begin{tikzpicture}[
			>=Latex,    
			thick,       
			scale=0.9,  
			every node/.style={font=\small, inner sep=0pt, outer sep=2pt} 
			]
			\coordinate (O) at (0,0); 
			
			\draw[->] (O) -- (-2.0, 1.45)    node[below left] {$t x- y$};
			\draw[->] (O) -- (-1.5,2.25)    node[below left] {$- y$};
			
			\draw[->] (O) -- (-0.55, -0.83) node[above left] {$tx$};
			\draw[->] (O) -- (-1.0,-1.5) node[below left] {$x$};
			
			\draw[->] (O) -- (-0.5, -2.6) node[below right] {$x+t y$};
			
			\draw[->] (O) -- (0.7,-1.05) node[below right] {$ y$};
			
		\end{tikzpicture}
	\end{center}
	\caption{Illustration of skew isosceles orthogonality
		}
	\label{fig:l_inf_l1}
\end{figure}	
\begin{example}
Let $X=\mathbb{R}^2$ endowed with the norm 
$$\|(x_{1},x_{2})\|=\max \{|x_{1}|,|x_{2}|\}.$$
Let $x=(0,1)$, $y=(1,0),$ it is clear that $\| x+t y\|=\|t x- y\|$.
\end{example}

 Based on the constants studied previously and the definition of skew isosceles orthogonality, we introduce a new constant \(D_t^B(X)\) and investigate some of its properties. This new constant characterizes the discrepancy between Birkhoff orthogonality and skew isosceles orthogonality.

Suppose that $X$ is a  Banach space with dimension at least 
2. We first present the key definitions as follows: for fixed $t>0$,
$$
D_{t}^{B}(X)=\sup \left\{\frac{\| x+t y\|-\|t x- y\|}{t}: x, y \in S_X, x \perp_B y\right\} .
$$

\begin{remark}
The new constant \(D_{t}^{B}(X)\) is also a generalization of 	$D^{\prime}(X)$ .
\end{remark}

\begin{lemma}\label{22} 
	Let  $X$  be a Banach space.  Then there exist unit vectors $x, y \in S_X$ such that $x \perp_B y$ and $-y \perp_B x$. 
\end{lemma}
\begin{proof}
According to the results on Birkhoff orthogonality in \rm{\cite{35}}, there exist unit vectors $x_0, y_0 \in X$ such that $x_0 \perp_B y_0$ and $y_0 \perp_B x_0$. Combined with the homogeneity of Birkhoff orthogonality, it follows that there exist $x, y \in S_X$ such that $x \perp_B y$ and $-y \perp_B x$.

\end{proof}

\begin{proposition}\label{2} 
	Let  $X$  be a Banach space. Then
	$$
	0 \leq D_{t}^{B}(X) \leq \frac{1}{t}.
	$$
\end{proposition}
\begin{proof}
From  Lemma \ref{22},  there exist unit vectors $x, y \in S_X$ such that  $x \perp_B y$ and $-y \perp_B x$. Consider the two pairs $(x, y)$ and $(-y, x)$,  we can obtain
	$$
	\frac{\|x + ty\| - \|tx - y\|}{t} \quad \text{and} \quad \frac{\|tx - y\| - \|x + ty\|}{t},$$
	which are negatives of each other. Therefore, at least one of them is non-negative, so we conclude  that $D_t^B(X) \ge 0.$
	
	On the other hand,
	 for any \( x, y \in S_X \) and  $x \perp_B y$, we can obtain that 
	$$
	\frac{\|x+t y\|-\|t x-y\|}{t} \leq \frac{1+t-t}{t}=\frac{1}{t},
	$$
as desired.
	
\end{proof}

When the dimension of the space is sufficiently large, the lower bound admits an alternative interpretation by virtue of the celebrated Dvoretzky's theorem, a profound result in functional analysis.
The values of geometric constants are determined by the structure of the space, and this theorem characterizes the subspace structure of high-dimensional Banach spaces.
For studies on geometric constants via Dvoretzky's theorem, we refer the reader to \cite{29},  \cite{323} .

\begin{lemma} (\cite{29} Theorem 10.43). \label{10.43}
	 For all $\varepsilon>0$, every infinite dimensional Banach space $X$ contains $\ell_2^n s(1+\varepsilon)-$ uniformly.
\end{lemma}
\begin{remark}
	Suppose that $x, y \in S_H$ with $x \perp_B y$, then we can obtain
	$$
	\begin{aligned}
		& \frac{\| x+t y\|-\|tx- y\|}{t} \\
		& =\frac{\sqrt{\| x+t y\|^2}-\sqrt{\|tx- y\|^2}}{t} \\
		& =\frac{\sqrt{\|x\|^2+t^2\|y\|^2+2 t\langle x, y\rangle}-\sqrt{t^2\|x\|^2+\|y\|^2-2 t\langle x, y\rangle}}{t}=0
	\end{aligned}
	$$
	so $D_{t}^{B}(H,X) = 0$. By Lemma \ref{10.43}, we have
	$$
	\begin{aligned}
		&D_{t}^{B}(X) \geq D_{t}^{B}(H,X)= 0,
	\end{aligned}
	$$
	for any infinite-dimensional Banach space.
	
\end{remark}

\begin{lemma} \cite{36}\label{3.2}
 Let $(X,\|\cdot\|)$ be a normed space, then $X$ is an inner product space if and only if the following property holds :

$$
x, y \in X,\|x\|=\|y\| \Rightarrow\|tx+y\|=\|x+t y\| \quad\forall t\in \mathbb{R} .
$$	   
\end{lemma}

\begin{proposition}  If $ X $ is an inner product space, then $D_{t}^{B}(X)=0$.
\end{proposition}
\begin{proof} 
Suppose that \(X\) is an inner product space. By Lemma \ref{3.2}, combined with the fact that Roberts orthogonality and Birkhoff orthogonality are equivalent in inner product spaces, we have
\[
\|x+t y\|-\|tx-y\|=\|x+t y\|-\|x-ty\|=0
\]
for  \(x, y \in S_{X}\), and hence \(D_{t}^{B}(X)=0\).

\end{proof}
\begin{example} 
	Let $ X=\mathbb{R}^{2}, t \geq 1$ and assign the following $ \ell_{\infty}-\ell_{1} $ norm
	$$
	\|x\|=\left\{\begin{array}{ll}
		\|x\|_{1}, & x_{1} x_{2} \leq 0 ,\\
		\|x\|_{\infty}, & x_{1} x_{2} \geq 0.
	\end{array}\right.
	$$
	Then $ D_{t}^{B}(X)=\frac{1}{t} $.
	
	Suppose that $x=(1, 0)$, $y=(1,1)$, it is easy to see that 
	 $x, y \in S_{X}$ and $x \perp_{B} y $. We can direct calculation gives
	$\|x+ty\|=1+t$ and $\|tx-y\|=t$.
	So we can obtain that
	$$
	\frac{\|x+t y\|-\|t x-y\|}{t}=\frac{1+t-t}{t}=\frac{1}{t}.
	$$
	
\end{example} 

\begin{theorem}
 The upper bound $\frac{1}{t}$ of $D_{t}^{B}(X)$ is attained by a pair of points of $S_X$ if and only if there exist two points $x, y \in S_X$ such that $[x, y]$ and $\left[x, x-\frac{2}{t} y\right]$ are both contained in $S_X$. Correspondingly, the length of the segment $[x, y]$ is at least $1$.
\end{theorem}
\begin{proof}
	Suppose that exists $x, y \in S_{X}$ such that  $x \perp_{B} y$, so we can have the follow inequality,
	$$
	\frac{\|x+ty\|-\|tx-y\|}{t}=\frac{1}{t} .
	$$
	
	Since
	$$
	\frac{1}{t}=\frac{\| x+ty\|-\|tx-y\|}{t} \leq \frac{1+t-t}{t}=\frac{1}{t},
	$$
we can obtain
	$
	\|x+ty\|=1+t
	$
	and
	$
	\|tx-y\|=t
	$.
	 
	Then we have  $$\left\|\frac{1}{1+t} x+\frac{t}{1+t} y\right\|=1,$$  hence $[x, y] \subset S_{X}.$
	
	On the other hand, let $\|tx-y\|=t$  hold for $t>0$, we have
	$$
	\left\|\frac{1}{2} x+\frac{1}{2} x-\frac{1}{2}\dfrac{2}{t}y\right\|=1,
	$$
	This means that the midpoint of $x$ and $x-\dfrac{2}{t}y$	has norm 1,
		 then we have $\left[x, x-\dfrac{2}{t}y\right] \subset S_{X} $. In addition, since $x \perp_{B} y$, we have $ \|x-y\| \geq 1$.
	
	On the contrary, $x, y \in S_{X} $ such that $ [x, y], \left[x, x-\dfrac{2}{t}y\right] \subset S_{X} $, then we have $ \|x+ty\|=1+t$ and $\|tx-y\|=t.$
	
	Hence
	$$
	\frac{\|x+ty\|-\|tx-y\|}{t}=\frac{1+t-t}{t}=\frac{1}{t}.
	$$
\end{proof}

Next, we establish the relationship between the geometric constant $D_{t}^{B}(X)$
and the modulus of convexity  $\delta_{X}(1) $ of the space $X$.

\begin{theorem}
		Let $X$  be Banach space, if $D_{t}^{B}(X)=\frac{1}{t}$ then $\delta_{X}(1)=0.$
\end{theorem}
\begin{proof}
	If  $D_{t}^{B}(X)=\frac{1}{t} $, we assume that $p \in \mathbb{N} $, then we have two points $ x_{p}, y_{p} \in S_{X}$  satisfying  $x_{p} \perp_{B} y_{p}$  and $ t_{p}>0 $ such that
	$$
	\frac{\left\| x_{p}+t_{p} y_{p}\right\|-\left\|t_{p} x_{p}-y_{p}\right\|}{t_{p}}>\frac{1}{t}-\frac{1}{p} .
	$$
	
Then we assume that $ t\geq 0$ such that  $t=\lim _{p \rightarrow \infty} t_{p}.$ Next we will distinguish three cases on $t.$
	
	Case 1:  $0 \leq t < 1.$ When  $p$  is sufficiently large we may assume that  $0 \leq t_{p}\leq1.$ From the following inequality
	$$
	\left\| x_{p}+t_{p} y_{p}\right\| \leq t_{p}\left\|x_{p}+y_{p}\right\|+\left(1-t_{p}\right)\left\|x_{p}\right\|.
	$$
	
	We have
	$$
	\begin{aligned}
		t_{p}\left\|x_{p}+y_{p}\right\| & \geq\left\| x_{p}+t_{p} y_{p}\right\|-\left(1-t_{p}\right)\left\| x_{p}\right\| \\
		& >\left(\frac{1}{t}-\frac{1}{p}\right) t_{p}+\left\|t_{p} x_{p}-y_{p}\right\|-1+t_{p} \\
		& \geq\left(\frac{1}{t}-\frac{1}{p}\right) t_{p}+t_p-1+t_{p} \\
		& =\left(\frac{1}{t}-\frac{1}{p}+2\right) t_{p}-1.
	\end{aligned}
	$$
	
	Hence $ \left\|x_{p}+y_{p}\right\|>\frac{1}{t}-\frac{1}{p}+2-\frac{1}{t_p}$ . Therefore we have
	$$
	\lim _{p \rightarrow \infty}\left(1-\frac{1}{2}\left\|x_{p}+y_{p}\right\|\right)=0,
	$$
	which implies that  $\delta_{X}(1)=0 $.
	
	Case 2:  $ t>1.$ When  $p$  is sufficiently large we may assume that  $ t_{p}>1.$ From the following inequality
	$$
	\left\| x_{p}+t_{p} y_{p}\right\| \leq \left\|x_{p}+y_{p}\right\|+\left(t_{p}-1\right)\left\|y_{p}\right\|.
	$$
	
	We have
	$$
	\begin{aligned}
		\left\|x_{p}+y_{p}\right\| & \geq\left\| x_{p}+t_{p} y_{p}\right\|-\left(t_{p}-1\right)\left\| y_{p}\right\| \\
		& >\left(\frac{1}{t}-\frac{1}{p}\right) t_{p}+\left\|t_{p} x_{p}-y_{p}\right\|+1-t_{p} \\
		& \geq\left(\frac{1}{t}-\frac{1}{p}\right) t_{p}+t_p+1-t_{p} \\
		& =\left(\frac{1}{t}-\frac{1}{p}\right) t_{p}+1.
	\end{aligned}
	$$
	 Therefore we have $\delta_{X}(1)=0 $.
	
	Case 3:  $t=1$, then we have $\left\{t_{p}\right\}_{p=1}^{\infty} $ contained either in  $[0,1]$  or in $(  1,+\infty  )$. By applying as employed in above the two cases , we have $\delta_{X}(1)=0$.
\end{proof}

\begin{proposition}
	Let $X$ be Banach space with $\delta_{X}(1)>0$ and $t \in (0,1] $, then $$D_{t}^{B}(X)<\frac{1-2 t \delta_{X}(1)}{t} .$$
\end{proposition}
\begin{proof}
	Suppose that $ x, y \in S_{X}$ and $ x \perp_{B}y$, then we have  $\| x- y\| \geq 1$, By the definition of  $\delta_{x}(1)$, we can obtain that $\|x+y\| \leq 2\left(1-\delta_{x}(1)\right)$.
	
So we can obtain that
	$$
	\begin{aligned}
		\| x+t y\|& \leq t\|x+y\|+1-t \\
		& \leq 2t(1-\delta x(1))+1-t .
	\end{aligned}
	$$
	
	Thus
	$$
	\begin{aligned}
		\frac{\| x+t y\|-\|t x-y\|}{t} & \leq \frac{2t(1-\delta x(1))+1-t-t}{t} \\
		& =\frac{1-2 t \delta_{X}(1)}{t}.
	\end{aligned}
	$$
\end{proof}

\begin{proposition} \label{4} 
	For any Banach space  $X$, $t\geq1$, we have
	$$
	D_{t}^{B}(X) \leq J^{B}(X)-1.
	$$
\end{proposition}
\begin{proof}
Assume that $t\geq1$,  combined with the  conditions that $ x, y \in S_{X}$ and $ x \perp_{B}y$, then we have
	$$
	\begin{aligned}
		&\frac{\|x+ty\|-\|tx-y\|}{t} \\
		&\leq \frac{\|x+y\|+|1-t|\|y\|-t}{t} \\
		&=\frac{\|x+y\|+t-1-t}{t} \\
		&\leq\|x+y\|-1,
	\end{aligned}
	$$
so we can obtain that	$$
	D_{t}^{B}(X) \leq J^{B}(X)-1.
	$$
\end{proof}
Note that a Banach space $X$ is uniformly non-square if and only if
$J^B(X)<2$ \cite{12}, by the above Proposition \ref{4}, we deduce the
following Corollary.

\begin{corollary} Let $X$ be a Banach space. Then $D_{t}^{B}(X)=1$  for $t \geq 1$
	if and only if $X$ is not uniformly non-square.
\end{corollary}
\begin{proof} 
	By the Proposition \ref{4}, 
	when  $D_{t}^{B}(X)=1$, we can deduce that
	$$
	1\leqslant J^B(X)-1 .
	$$
	This shows that \( J^B(X) \geq 2 \). It should be noted that the upper bound is 2, and combined with the conclusion that \( J^B(X) = 2 \) if and only if the space \( X \) is not uniformly nonsquare \cite{6}, as desired.
\end{proof}

\section{Some applications of $A^B_{t}(X)$ and $D_{t}^{B}(X)$}

Recall that an orthogonality relation $ \perp $
is calls symmetric if the condition $x\perp y$
necessarily entails  $\perp x$
. It is straightforward to verify that the canonical orthogonality defined on Hilbert spaces is symmetric. In contrast, it is a well-documented fact that Birkhoff orthogonality lacks this symmetry property. Nevertheless, James established the following conclusion in the work \rm{\cite{24}}\label{L1}.

\begin{figure}[htbp]
	\centering
	\begin{tikzpicture}[scale=1.5, >=Stealth]
		
		\draw (-1.5,0) -- (1.5,0) ;
		\draw (0,-1.5) -- (0,1.5);

		\draw[very thick, opacity=0.7] (1,1)  -- (0,1) -- (-1,0) --(-1,-1) --(0,-1) --(1,0) -- cycle;
	\end{tikzpicture}
	\caption{Geometric explanation of $\ell_\infty-\ell_1$.}
	\label{fig:2_inf_l1}
\end{figure}
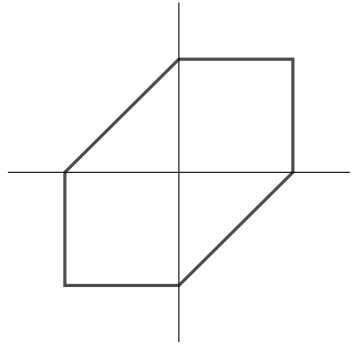	

\begin{lemma} \rm{\cite{24}}\label{L2}
A Banach space $X$ with dimension $\operatorname{dim} X \geq 3$ is a Hilbert space if and only if Birkhoff orthogonality is symmetric on $X$.
\end{lemma}

It is worth emphasizing that the space dimension assumption in the previous lemma cannot be discarded. We first revisit the definition of Radon planes, which is essential to our analysis.

\begin{definition} \rm{\cite{25}}
	A two-dimensional Banach space in the Birkhoff orthogonality is symmetric, defined as a Radon plane.
\end{definition}

From Proposition \ref{P2} and the fact that $X$ is a Radon plane, we obtain
$
A_2(X, B) \leq \frac{3}{2},
$
which yields the following conclusion.

\begin{corollary}\label{1.1}
	Suppose that $X$ is a Radon plane, then $A^B_{t}(X) \leq 1+t-\frac{1}{2}\min\{1,t\}.$
\end{corollary}

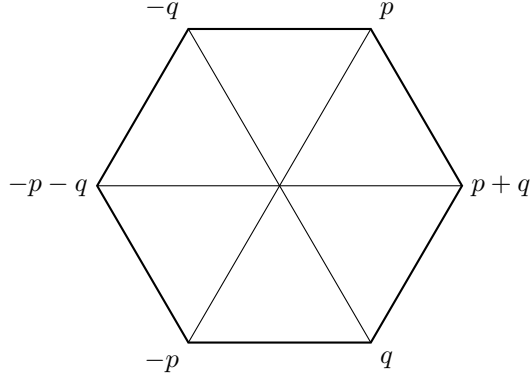
\begin{figure}[htbp]
	\centering
	\begin{tikzpicture}[scale=1.2, >=Stealth]
		
		\coordinate (p) at (2,0);
		\coordinate (pq) at (1,1.732);
		\coordinate (q) at (-1,1.732);
		\coordinate (minus_p) at (-2,0);
		\coordinate (minus_pq) at (-1,-1.732);
		\coordinate (minus_q) at (1,-1.732);
		
		\draw[thick] (p) -- (pq) -- (q) -- (minus_p) -- (minus_pq) -- (minus_q) -- cycle;
		
		\draw[thin] (p) -- (minus_p);
		\draw[thin] (q) -- (minus_q);
		\draw[thin] (minus_pq) -- (pq);
		\node[right] at (p) {$p+q$};
		\node[above right] at (pq) {$p$};
		\node[above left] at (q) {$-q$};
		\node[left] at (minus_p) {$-p-q$};
		\node[below left] at (minus_pq) {$-p$};
		\node[below right] at (minus_q) {$q$};

	\end{tikzpicture}
	\caption{The unit sphere of the affine regular hexagon}
	\label{fig:affine_hexagon}
\end{figure}	

By an affine-regular hexagon, any non-degenerate affine image obtained from a regular hexagon via affine transformation \rm{\cite{26}}. We observe that the space $\mathbb{R}^{2}$ with the norm $ \ell_{\infty}-\ell_{1} $ is a Radon plane \rm{\cite{24}}\label{L3} (see Figure 2) and its unit sphere is an affine-regular hexagon. we have below example to achieve the upper bound of the corollary.

\begin{example}
	Suppose that $X$ is a Radon plane $ \ell_{\infty}-\ell_{1} $, the space $\mathbb{R}^{2}$ with the norm defined by
	$$
	\|(x_1, x_2)\| = \begin{cases} \|(x_1, x_2)\|_\infty, & x_1x_2 \geq 0, \\ \|(x_1, x_2)\|_1, & x_1x_2 \leq 0. \end{cases}
	$$
	Then $A^B_{t}(X)=1+t-\frac{1}{2}\min\{1,t\}.$
\end{example}

\begin{proof}	
	Suppose that $x=(1,0)$ and $y=(0,1)$. We can infer that $x, y \in S_X$ such that $x \perp_B y$.
	In addition, when $t\geq1$, we have $\|x+ty\|=1$, $\|tx-y\|=t+1.$ 
	Thus, it follows from Corollary \ref{1.1} that we derive
$A^B_{t}(X)=1+\frac{t}{2}$.
		
	Similary, when $0\leq t\leq 1$, we have $\|x+ty\|=t$, $\|tx-y\|=t+1,$ so we can obtain that  $A^B_{t}(X)=\frac{1}{2}+t$.
	
	Thus we can complete the example.

\end{proof}

Next, we present a definitive result that linking the upper bound of $A^B_{t}(X)$ on an arbitrary Radon plane to the condition that the unit sphere is an affine-regular hexagon (see Figure 3). 

\begin{theorem} \label{15}
Suppose that $X$ is a Radon plane. Then the equality $A^B_{t}(X)=1+t-\frac{1}{2}\min\{1,t\}$ 
holds if and only if the unit sphere $S_X$ of $X$
is an affine-regular hexagon.
\end{theorem}
\begin{proof}
	For any Radon plane, if we suppose that $A^B_{t}(X)=1+t-\frac{1}{2}\min\{1,t\}$, then from Proposition \ref{P2} we have $A_2(X,B)\geq\frac{3}{2}$ and the conclusion that  $ A_2(X,B)\leq \frac{3}{2}$ in Radon plane, we can obtain that  $A_2(X,B)=\frac{3}{2}$. Therefore $S_X$ is an affine-regular hexagon.
	
	When $0\leq t \leq 1$, suppose that an affine-regular hexagon is $S_X$ , we have distinct points $ p, q \in S_X $ such that $\pm p, \pm (p+q), \pm q $ are the vertices of $S_X$ (see Figure 3). Take $x=p+q$ and $y =p.$ Then it is evident that $ x, y \in S_X $ with $x \perp_B y $.
	
	In addition, we are also able to obtain
	$$
	\|x+ty\|=\|(1+t)p+q\|=(1+t)\Big\|\frac{t}{1+t}p+\frac{1}{1+t}(q+p)\Big\|=1+t,
	$$
	$$
	\|tx-y\|=\|tq-(1-t)p\|=\|tq+(1-t)(-p)\|=1.
	$$
	Hence, by the definition of $A^B_{t}(X)$, we have
	$$
	1+\frac{t}{2}=\frac{\|x+ty\|+\|tx-y\|}{2}\leq A^B_{t}(X)\leq 1+\frac{t}{2},
	$$
	which means that $A^B_{t}(X)=1+\frac{t}{2}$.
	
	Similary, when $t\geq 1$, we can deduce that $A^B_{t}(X)=t+\frac{1}{2}$. This completes the proof.
\end{proof}

\begin{remark}
	Let $X$ be the space $\mathbb{R}^2$ endowed with the norm
	$$
	\|x\|=\max \left\{\left|x_1\right|,\left|x_2\right|, \frac{\left|x_1\right|+\left|x_2\right|}{\sqrt{2}}\right\},
	$$
	note that the space $X$ is not an inner product space with $A_2(X,B)=\sqrt{2}$ (\cite{5}).
	
	And from Proposition \ref{P2} we have $$
	A^B_{t}(X) \geqslant \min \{1, t\}A_2(X,B).
	$$ 
	When $$A^B_{t}(X) =\sqrt{2} \min \{1, t\},$$ we can obtain  $A_2(X,B)\leq\sqrt{2}$, that is, $A_2(X,B)=\sqrt{2}$.
	
	So, when $A^B_{t}(X) = \sqrt{2} \min \{1, t\}$, we are currently unable to find a suitable method to determine whether a space is an inner product space.
	
\end{remark}\label{222222} 
The following result is interesting: when we lift the orthogonality condition on the constant, the new constant is able to characterize the inner product space.
\begin{remark}
	For $t>0$, if we set $$A^{\prime }_{t}(X)=\sup\Big\{\frac{\|x+t y\|+\|tx- y\|}{2}: x,y \in S_X \Big\},$$
	we can obtain that $$ A^{\prime }_{t}(X)\geq \min \{1, t\} A_2(X).$$
	If $A_t^{\prime }(X)=\sqrt{2} \min\{1, t\}$, we can deduce that $A_2(X) \leq \sqrt{2}$. 
	
	Thus, we have  $A_2(X)=\sqrt{2}$, and hence $J(X)=\sqrt{2}$. So, when $\operatorname{dim} X \geqslant 3$. By \cite{16}, it can be deduced that it is an inner product space.
	
	We can restate the above result, which characterizes inner product spaces via a constant, in the following inequality form.
	
	For any fixed $t>0$, if
	$$
	\|x+t y\|+\|t x-y\| \leq 2 \sqrt{2} \min \{1, t\}, \hfill \text{for all} \ x,y \in S_X,
	$$ 
	then we can conclude that $X$ be an inner product space.
	This result can be analogized to Lemma \ref{3.2}.
\end{remark}

\begin{remark}

In this paper, our geometric constants are only considered for the case involving a single parameter \(t\). For the two-parameter case, similar arguments apply, and we can similarly establish their upper and lower bounds.
Most of the other results remain valid as well.

\end{remark}

At the end of the article, we would like to pose some open questions:

\text { Problem } 1: 

Due to the properties of Birkhoff orthogonality, it is an interesting topic to incorporate the classical von Neumann-Jordan constant into the framework of Birkhoff orthogonality and redefine it as follows:
\[
C_{\mathrm{NJ}}(X)=\sup \left\{\frac{\|x+y\|^2+\|x-y\|^2}{2\left(\|x\|^2+\|y\|^2\right)}: x, y \in X,\ (x,y)\neq(0,0),\
x \perp_B y\right\}.
\]
Of course, the estimation of related inequalities also presents considerable difficulties. Furthermore, it is also natural to incorporate the \(p\)-th von Neumann-Jordan constant \cite{3333} into Birkhoff orthogonality. For related results on the \(p\)-th Birkhoff orthogonality, we refer the reader to \cite{4444}.

\text { Problem } 2:  

Remark \ref{222222} motivates us to investigate whether we can improve the proof of the well-known Lemma \ref{3.2} by defining the constant as follows: set
\[
F(X)=\sup_{t>0}\left\{\frac{\|x+t y\|-\|tx+ y\|}{2}: x,y \in S_X \right\}.
\]

\bmhead{Acknowledgements}

Thanks to all the members of the Functional Analysis Research Team at the School of Mathematics and Statistics, Anqing Normal University, for their valuable discussions and corrections regarding the
challenges and errors encountered in this article. 
We also thank Dr. Dandan Du for her valuable comments on this paper.

\end{document}